\documentclass[10pt]{scrarticle}

\usepackage[utf8]{inputenc}
\usepackage[a4paper]{geometry}
\usepackage[colorlinks=true,linkcolor=teal,citecolor=teal,hypertexnames=false]{hyperref}
\usepackage{amsmath,amssymb,amsthm}
\usepackage{enumitem}
\usepackage{comment}
\usepackage{microtype}
\usepackage[english]{babel}
\usepackage{mathrsfs}
\usepackage{mathtools} 
\usepackage{stmaryrd}
\usepackage{soul} 

\usepackage{tikz}
\usetikzlibrary{matrix,arrows}
\usepackage{tikz-cd}
\usetikzlibrary{cd}
\usepackage{comment}

\newcounter{parnum}[section]

\usepackage{calrsfs}
\DeclareMathAlphabet{\pazocal}{OMS}{zplm}{m}{n}

\mathcode`A="7041 \mathcode`B="7042 \mathcode`C="7043 \mathcode`D="7044
\mathcode`E="7045 \mathcode`F="7046 \mathcode`G="7047 \mathcode`H="7048
\mathcode`I="7049 \mathcode`J="704A \mathcode`K="704B \mathcode`L="704C
\mathcode`M="704D \mathcode`N="704E \mathcode`O="704F \mathcode`P="7050
\mathcode`Q="7051 \mathcode`R="7052 \mathcode`S="7053 \mathcode`T="7054
\mathcode`U="7055 \mathcode`V="7056 \mathcode`W="7057 \mathcode`X="7058
\mathcode`Y="7059 \mathcode`Z="705A

\newcommand\FF{\mathbb F}

\newcommand\CI{\overline{K}_{\infty}}
\newcommand\RR{\mathbb R}

\newcommand\p{\mathfrak p}
\newcommand\Fp{\FF_{\p}}

\DeclareMathOperator{\coker}{coker}

\newcommand\End{\operatorname{End}}
\newcommand\id{\operatorname{id}}

\newcommand\Lie{\mathrm{Lie}}
\newcommand\Hom{\mathrm{Hom}}
\newcommand\Mod{\mathbf{Mod}}

\newcommand\Fitt{\mathrm{Fitt}}

\newcommand\llangle{\langle\!\langle}
\newcommand\rrangle{\rangle\!\rangle}

\newtheorem{thm}[parnum]{Theorem}
\newtheorem*{thm*}{Theorem}
\newtheorem{lem}[parnum]{Lemma}

\newtheorem{conj}[parnum]{Conjecture}
\newtheorem{prop}[parnum]{Proposition}
\newtheorem{cor}[parnum]{Corollary}
\theoremstyle{definition}
\newtheorem{deftn}[parnum]{Definition}
\newtheorem{cons}[parnum]{Construction}
\newtheorem*{deftn*}{Definition}
\newtheorem{rem}[parnum]{Remark}
\newtheorem*{rem*}{Remark}

\let\liminf\relax
\DeclareMathOperator*{\liminf}{liminf}

\title{A class formula for\\effective Anderson motives}

\author{%
Xavier Caruso\footnote{CNRS, Université de Bordeaux, IMB, \emph{Email:} \texttt{xavier@caruso.ovh}},
Quentin Gazda\footnote{Sorbonne Universit\'e and Universit\'e Paris Cit\'e, CNRS, IMJ-PRG, \emph{Email:} \texttt{quentin@gazda.fr}},
Alexis Lucas\footnote{Université de Caen, LMNO, \emph{Email:} \texttt{alexis.lucas@unicaen.fr}}}

\begin{document}

\maketitle

\setcounter{tocdepth}{1}

\begin{abstract}
We establish a class formula with variable for arbitrary effective 
Anderson motives. This pursues the pioneering work of Taelman and 
extends the scope of the work of Anglès--Ngo Dac--Tavares Ribeiro, 
removing the admissibility assumption and allowing for more general
Anderson modules. 
In the same setting, we further conjecture a ``nondeterminantal''
version of the class formula which holds before taking determinants
and prove a weak version of it.
Finally, we prove an overconvergent version of the class formula and
derive from it strong convergence properties for the associated 
$L$-series.
\medskip

\noindent
\textsc{MSC~2020.} Primary 11G09; Secondary 11M38, 11R58.
\end{abstract}

\tableofcontents

\bigskip

In a seminal article~\cite{taelman}, 
Taelman proves a class formula expressing the value at $1$ of the
$L$-series of a Drinfeld module $\phi$ as the product of the Fitting
ideal of the class module of $\phi$ and its regulator, defined by
means of the module of so-called Taelman units.
This formula was then extended by Anglès, Ngo Dac and Tavares Ribeiro
in~\cite{ANT} in two directions. First, their formula is more general
in the sense that it not only provides the value at $1$ but the
complete $L$-series. Second, the setting is more general: instead
of Drinfeld modules, the authors of~\cite{ANT} work with Anderson
modules and manage to prove a class formula for almost all of them,
the so-called \emph{admissible} ones.

In this short article, we revisit the work of~\cite{ANT} and
extend it further in several directions.
Let $\FF$ be a finite field, $C$ be a smooth projective curve over $\FF$ and $\infty$ be a closed point on $C$. 
We denote by $A$ the ring of algebraic functions on $C$ that are regular away from $\infty$:
\[
A=\mathcal{O}_C(C\setminus \{\infty\})
\]
and we denote by $K$ its fraction field or, equivalently, the function field of $C$. Let $L$ be a finite extension of $K$ and let $R$ be the integral closure of $A$ in $L$.
In this article, we work in the setting of effective Anderson $A$-motives over $R$
(see Definition~\ref{def:motive}); this encompasses the Anderson modules that are Zariski locally isomorphic to $\mathbb{G}_a^d$ (and not merely isomorphic to $\mathbb{G}_a^d$, as considered in~\cite{ANT}).

We first extend to our framework all the constructions which are needed to state 
the class formula: we define \emph{the $L$-series} $L(M; \tau)$ and \emph{the module of 
$\tau$-units} $U_M \subset \Lie_M(L_\infty\langle\tau\rangle)$ attached to an 
effective Anderson motive $M$.
In what precedes, $L_\infty\langle\tau\rangle$ is the topological module of Tate series in
the variable $\tau$ with coefficients in the ring $L_\infty$, defined as the
completion of $L$ at the point $\infty$.
The attentive reader certainly observed that we used the greek letter $\tau$
for the variable, while the classical notation is $z$.
The reason why we prefer $\tau$ is because this variable is closely related 
to the Frobenius acting on $L$ or $L_\infty$, as it will become apparent in
this article.

With this preparation and writing $K_\infty$ for the completion of $K$ at
$\infty$, our main theorem is the following.

\begin{thm*}
Let $M$ be an effective Anderson $A$-motive over $R$.
Then $U_M$ is a full-rank projective $A[\tau]$-lattice in 
$\Lie_M(L_{\infty}\langle\tau\rangle)$ and we have an equality
on $A[\tau]$-lines
\[
\mathrm{det}_{A[\tau]} \big(U_M\big) = L(M;\tau) \cdot \mathrm{det}_{A[\tau]}\big(\Lie_M(R[\tau])\big)
\]
in $\mathrm{det}_{K_{\infty}\langle\tau\rangle} \big(\Lie_M(L_{\infty}\langle\tau\rangle)\big)$.
\end{thm*}

Compared to~\cite{ANT}, the most important advantages of our
theorem are the following.
First, the setting is more general as already mentionned above,
and we no longer need any admissibility assumption.
Second, our theorem is more precise: it provides an equality
of $A[\tau]$-lines whereas the main theorem of \cite{ANT} only
compares the lines they generate over 
$A(\tau) := \FF(\tau) \otimes_{\FF[\tau]} A[\tau]$.
While it is more or less possible to move from the latter to the
former by introducing sign functions, our result does not need this 
extra step and, for this reason, it is easier to handle and more
robust.

A third advantage of our approach is related to our method of
proof which is, we think, more intrinsic, less technical and so
possibly easier to generalize to other settings, \emph{e.g.}
noneffective Anderson motives.
While it reuses several important ingredients introduced and/or
developed in~\cite{taelman} and~\cite{ANT}, it also introduces
new ideas and arguments.
In particular, we do not use the notion of $z$-deformation but
reinterpret it at the level of the comotive.
This makes all constructions more functorial and easier to handle
in many cases.
Moreover, a significant part of our proof consists in closely
studying the algebraic structure of $U_M$; this eventually relies
on standard but nontrivial theorems of commutative algebra over 
polynomial rings, including the Quillen--Suslin theorem.
In addition of being interesting on its own, this preliminary
part allows for a drastic simplification and an extension of the
scope of the final argument.

Also, when working directly over $A[\tau]$ looks untractable,
we always prefer extending scalars to $K[\tau]$ instead of $A(\tau)$
as it was systematically done in~\cite{ANT}. This has several
practical advantages (in particular it is an important idea to
relax the admissibility assumption) but also, and above all,
it has a strong arithmetical content.
Indeed inverting the elements of $A$ somehow corresponds to
working up to isogeny, while inverting polynomials in 
$\FF[\tau]$ does not have a clear interpretation (at least
to us).
After this, we believe that studying Anderson motives up to 
isogeny could be quite interesting on its own,
and we plan to come back to it in a subsequent article.

Last but not least, our approach suggests further possible
improvements. In this perspective, we establish a direct relationship 
between $\Lie_M(R[\tau])$ and $U_M$, which holds before taking
determinants (see Theorem~\ref{thm:nondet:classformula}).
We expect that such a formula could be useful, for example, 
to study equivariant versions of the class formula as in~\cite{CFM}.

Finally, in Section~\ref{sec:convergence}, we prove an
overconvergence theorem of Taelman units and show
that it implies extremely rapid convergence properties (that
we call \emph{superconvergence}) for the $L$\nobreakdash-series itself.
We conjecture an even stronger version of this result and
relate it to Taelman's conjecture~\cite{taelman-conj}.

\paragraph{Acknowledgments.}

The authors are deeply grateful to the organizers of the 
thematic programme ``Arithmetic of global function fields''
(Madrid, 2026) for inviting us to give a series of lectures
on the topic of the present article. They thank in particular
Bruno Anglès for his constant interest in this project and all 
his relevant comments. They also want to thank Andreas Maurischat 
for pointing out a simplification in the final argument of the
proof of Theorem~\ref{thm:classformula}.

The first author was supported by the ANR projects PadLefan
(ANR-22-CE40-0013) and PPAL (ANR-25-CE40-4664),
the second by the ANR project Al$K$traZ (ANR-25-CE40-3307-01).

\section{Anderson motives and their functor of points}
\label{sec:motives}

Let $\FF$ be a finite field with cardinality $q$ and characteristic $p$.
Let $C$ be a smooth projective curve over $\FF$ and fix $\infty$ as a 
closed point on $C$. We let $A$ denote the ring of functions on $C$ that 
are regular away from the point $\infty$:
\[
A=\mathcal{O}_C(C\setminus\{\infty\}).
\]
We denote by $K$ the function field of $C$ (equivalently, the fraction 
field of $A$). Let $L$ be a finite extension of $K$ and $R$ be its 
ring of integers, \emph{i.e.} the subring of elements of $L$ which are
integral over $A$.

We let $K_\infty$ be the completion of $K$ at $\infty$ and set 
$L_\infty := L \otimes_K K_\infty$. We underline that $L_\infty$ is 
not necessarily a field, but it is a finite product of finite 
extensions of $K_\infty$. Both $K_\infty$ and $L_\infty$ are
endowed with the $\infty$-adic norm, which we denote by
$\vert \cdot \vert$.

We form the Ore polynomial ring $R[\tau]$ whose elements are
polynomials in $\tau$ over $R$ with multiplication driven by the
rule $\tau r = r^q \tau$ for $r \in R$.
We also consider the tensor product $A \otimes R[\tau]$ (where
throughout the article, unlabeled tensor products are always taken
over $\FF$). We stress that $\tau$ commutes with elements in $A$.
In particular, $A[\tau]$ is a commutative subring over $A \otimes
R[\tau]$.

The relevant definition in our study is the following.
\begin{deftn}[Anderson motives]
\label{def:motive}
An \emph{Anderson $A$-motive over $R$} is a left module $M$ over 
$A \otimes R[\tau]$ such that:
\begin{itemize}
\item $M$ is finitely generated projective over $R[\tau]$,
\item for all $a \in A$, the multiplication by
$a \otimes 1 - 1 \otimes a$ is nilpotent on $\Lie(M) := M/R[\tau]\tau M$.
\end{itemize}
\end{deftn}

\begin{rem}[Anderson modules]\label{rem:anderson-modules}
An Anderson module over $R$ of dimension $d$ is a smooth and affine $A$-module scheme $E$ over $R$ whose underlying $\FF$-module scheme is a flat form of $\mathbb{G}_a^d$ and having the property that $a\otimes 1-1\otimes a$ acts nilpotently on $\Lie_E(R)$ for all $a\in A$ (see \cite[Definition~1.2]{hartl}). If $E$ splits as a form of $\mathbb{G}_a^d$ in the Zariski topology (and not merely the flat topology), then this matches Fang's definition \cite[Remark~1.4]{fang}.

There is a well-known equivalence between Anderson motives in the sense of Definition~\ref{def:motive} which are free of rank $d$ over $R[\tau]$ and Anderson modules of dimension $d$ which are split over $R$ (usually given in term of an $\FF$-algebra homomorphism $\phi:A \to M_d(R)[\tau]$). Definition \ref{def:motive} is more general since the projectivity assumption is weaker than freeness. In particular, it encompasses Anderson modules \emph{à la Fang} which are Zariski forms of $\mathbb{G}_a^d$; this follows from \cite[Remark~3.6]{gazda-maurischat}.

It does not \emph{a priori} encompass Hartl's more general definition \cite[Definition~1.2]{hartl}. Over a finite Dedekind extension $R$ of $A$, however, we are not aware of an Anderson module of dimension $d$ whose underlying $\FF$-vector space scheme is not a Zariski form of $\mathbb{G}_a^d$. 
\end{rem}

\subsection{The functor of points}
In the case where an Anderson motive $M$ arises from an Anderson module $E$ (Remark~\ref{rem:anderson-modules}), the latter can be recovered from the former by a canonical isomorphism (\emph{e.g.} \cite[Proposition~3.2]{gazda-maurischat})
\begin{equation}\label{eq:relation-motive-module}
E(S) \cong \Hom_{R[\tau]}(M,S).
\end{equation}
Above, $S$ is an $R$-algebra and homomorphisms are those of left $R[\tau]$-modules, treating $S$ as a left $R[\tau]$-module by the rule $\tau x = x^q$ for $x \in S$. 

We extend this notation to general Anderson motives and general left $R[\tau]$-modules.
\begin{deftn}[Functor of points]
To an Anderson motive $M$, we attach its \emph{functor of points}
$E_M$ defined by
$$\begin{array}{rcl}
E_M : \quad \Mod_{R[\tau]} & \longrightarrow & \Mod_A \smallskip \\
X & \longmapsto & \Hom_{R[\tau]}(M, X),
\end{array}$$
where the structure of $A$-module on $\Hom_{R[\tau]}(M, X)$ comes
from the $A$-action of $M$, \emph{i.e.} for $a \in A$ and $f : M
\to X$, the map $af$ is $x \mapsto f(ax)$.

Moreover, if $X$ is a bimodule over a pair of $\FF$-algebras $(R[\tau], B)$, we endow $\Hom_{R[\tau]}(M, X)$ with the inherited structure of right
$B$-module. If the action of $\FF\subset R[\tau]$ agrees with that of $\FF\subset B$, as it usually does, $E_M(X)$ then becomes a right module over $A \otimes B$.
\end{deftn}

When $M$ comes from an Anderson module $E$, the functor of points recovers several classical constructions.
\begin{itemize}
\item If $X = S$ is an $R$-algebra, we have canonically $E_M(S)\cong E(S)$ as $A$-modules by Equation~\eqref{eq:relation-motive-module}.
\item If $X = S$ is a $R$-algebra, endowed with the structure
of $R[\tau]$-module given by $\tau x = 0$ for $x \in S$, then
$E_M(S)$ coincides with $\Lie_E(S)$ by \cite[Lemma~1.3.4]{anderson}. 
\item If $X = R[\tau]$, then $E_M(R[\tau])$ inherits a structure of
right module over $A \otimes R[\tau]$; it is isomorphic the so-called
\emph{dual motive} or \emph{comotive} of $E$~\cite[Proposition~3.2]{gazda-maurischat}.
\end{itemize}

Following the second bullet point, we introduce another notation.
\begin{deftn}[Lie functor]
The \emph{Lie functor} of an Anderson motive $M$ is
$$\begin{array}{rcl}
\Lie_M : \quad \Mod_{R} & \longrightarrow & \Mod_A \smallskip \\
X & \longmapsto & \Hom_{R}(\Lie(M), X)
\end{array}$$
where we recall that $\Lie(M) = M/R[\tau]\tau M$.
\end{deftn}

We record the two following immediate properties, which will be
used repeatedly throughout this article.

\begin{lem}
Let $X$ be a $R[\tau]$-module on which $\tau$ acts by zero.
Then $E_M(X) = \Lie_M(X)$.
\end{lem}

\begin{prop}
The functors $E_M$ and $\Lie_M$ are exact.
\end{prop}

\begin{rem}[Base change]
We underline that the functors $E_M$ and $\Lie_M$ are
compatible with base change in the following sense. If $S$ be
an algebra over $R$, the tensor product $S[\tau] \otimes_{R[\tau]}
M$ defines an Anderson $A$-motive $M_S$ over $S$ and, whenever
$X$ is a module over $S[\tau]$, we have by adjunction
\begin{equation}\label{eq:base-change}
E_{M_S}(X) = 
\Hom_{S[\tau]}\big(S[\tau] \otimes_{R[\tau]} M, X\big) =
\Hom_{R[\tau]}\big(M, X\big) = E_M(X)
\end{equation}
and, similarly, $\Lie_{M_S}(X) = \Lie_{M}(X)$ for all $S$-module $X$.
\end{rem}

\begin{rem}[Choice of coordinates]\label{rem:choice-of-basis}
Because $L$ is a field, $L[\tau]$ is a left principal ideal domain. As a consequence, any finite projective left $L[\tau]$-module is free. In particular, the base change $M_L:=L[\tau]\otimes_{R[\tau]}M$ is free over $L[\tau]$ and (by Remark \ref{rem:anderson-modules}) comes from an Anderson module $E$ over $L$ whose underlying $\FF$-vector space scheme is isomorphic $\mathbb{G}_a^d$, where $d$ is the rank of $M_L$. After choosing a basis of $M_L$ (called \emph{a choice of coordinates for $M_L$}), and whenever $X$ factors as an $L[\tau]$-module, we obtain an isomorphism
\[
E_M(X)=E(X)\cong X^d.
\] 
\end{rem}

\subsection{Analytic noncommutative series}
We consider the noncommutative ring $L_{\infty}[\![\tau]\!]$ of formal power series in $\tau$ with coefficients in $L_{\infty}$. Many series related to the arithmetic of Anderson modules (such as exponentials and logarithms) naturally belong to this ring, although they often satisfy strong convergence properties. As a preliminary of what follows, we take some time to develop the basics of the theory of analytic series in the variable $\tau$.

\begin{deftn}[Noncommutative convergent power series]
For a positive real number $r>0$, we set
$$L_\infty\langle\tau; r\rangle := \Big\{ \sum_i a_i \tau^i \,\big|\, a_i \in L_\infty,\,
\lim_{i \to \infty} \vert a_i \vert r^{q^i} = 0\Big\} \subset L_\infty[\![\tau]\!]$$
and endow it with the Gauss norm defined by:
$$\Big\Vert {\textstyle \sum_{i \geq 0} a_i \tau^i} \Big\Vert_r := 
\max_{i \geq 0} \vert a_i \vert r^{q^i} \qquad (a_i \in L_\infty).$$
\end{deftn}

From an analytic point of view and interpreting $\tau$ as the
Frobenius $x \mapsto x^q$, the topological module $L_\infty\langle\tau; r\rangle$ is the space of converging analytic $\FF$-linear functions on the
closed disc of radius $r$ and $\Vert \cdot \Vert_r$ is the sup
norm on this disc.
When $r = 1$, we omit it in the notation: we simply write
$L_\infty\langle\tau\rangle$ for $L_\infty\langle\tau; 1\rangle$
and $\Vert \cdot \Vert$ for $\Vert \cdot \Vert_1$.
We also define the space of entire $\FF$-linear functions by
$$L_\infty\langle\tau; \infty\rangle := 
\bigcap_{r > 0} L_\infty\langle\tau; r\rangle.$$

\begin{lem}
\label{lem:rdcproduct}
Let $r, s \in \RR_{> 0}$.
If $f \in L_\infty\langle\tau; r\rangle$ and
$g \in L_\infty\langle\tau; s\rangle$ with $\Vert g \Vert_s \leq r$,
then $fg \in L_\infty\langle\tau; s\rangle$.
\end{lem}

\begin{proof}
We write $f = \sum_i a_i \tau^i$ and $g = \sum_j b_j \tau^j$.
Then $fg = \sum_{i,j} a_i b_j^{q^i} \tau^{i+j}$.
Let $\varepsilon$ be a real number in the range $(0, \Vert f \Vert_r)$.
By definition, there exist $i_0$ and $j_0$ such that
$$\vert a_i \vert r^{q^i} \leq \varepsilon
  \quad \text{and} \quad
  \vert b_j \vert s^{q^j} \leq \frac{\varepsilon r}{\Vert f \Vert_r}$$
whenever $i \geq i_0$, $j \geq j_0$.
Let us assume that $i + j \geq i_0 + j_0$.
If $i \geq i_0$, we can write
$$\vert a_i b_j^{q^i} \vert s^{q^{i+j}} = 
  \vert a_i \vert \cdot \big( \vert b_j \vert s^{q^j} \big)^{q^i} \leq
  \vert a_i \vert r^{q^i} \leq \varepsilon$$
the first inequality in the above line coming from the assumption
$\Vert g \Vert_s \leq r$. Otherwise, we necessarily have $j \geq j_0$
and we deduce
$$\vert a_i b_j^{q^i} \vert s^{q^{i+j}} \leq
  \vert a_i \vert \cdot \left(\frac{\varepsilon r}{\Vert f \Vert_r}\right)^{q^i} \leq
  \vert a_i \vert r^{q^i} \cdot \left(\frac{\varepsilon}{\Vert f \Vert_r}\right)^{q^i} \leq
  \Vert f \Vert_r \cdot \frac{\varepsilon}{\Vert f \Vert_r} = \varepsilon.$$
We conclude that $\vert a_i b_j^{q^i} \vert s^{q^{i+j}}$ goes to $0$
when $i{+}j$ goes to infinity, showing that $fg \in L_\infty\langle\tau; 
s\rangle$.
\end{proof}

\begin{rem}
Reasoning in terms of analytic functions, Lemma~\ref{lem:rdcproduct}
becomes crystal clear: since multiplication in $L_\infty[\![\tau]\!]$ 
corresponds to composition of functions, it simply asserts that if $g$ 
is analytic on some disc and bounded in norm by $r$ on this disc, 
then $fg$ (that is, $f \circ g$) is
analytic on the same disc provided that $f$ converges on the disc
of radius $r$.
\end{rem}

\begin{rem}
It follows from Lemma~\ref{lem:rdcproduct} that, for
any finite or infinite radius $r$, $L_\infty\langle\tau; r\rangle$ 
is a left module over $L_\infty\langle\tau; \infty\rangle$.
In particular, $L_\infty\langle\tau; \infty\rangle$ is stable under
multiplication, \emph{i.e.}, it is a ring. One should be extremely
careful however that $L_\infty\langle\tau; r\rangle$ are not rings
when $r$ is finite.

Nonetheless $L_\infty\langle\tau; r\rangle$ is a left module over
$R[\tau]$; it thus definitely can serve as an argument for the
functors $E_M$ and $\Lie_M$.
Moreover, $E_M(L_\infty\langle\tau\rangle)$
and $\Lie_M(L_\infty\langle\tau\rangle)$ inherit a right
action of $\tau$, turning them into $A[\tau]$-modules.
\end{rem}

\begin{cons}[Norm and topology]
One can endow these function spaces with norms and topologies as follows. After choosing coordinates for $M_L$ (Remark \ref{rem:choice-of-basis}), we have an
$A[\tau]$-linear isomorphism $E_M\big(L_\infty\langle \tau\rangle)
\simeq L_\infty\langle \tau\rangle^d$, which allows us to extend
the Gauss norm on $E_M\big(L_\infty\langle \tau\rangle)$.
We define similarly a norm on $\Lie_M\big(L_\infty\langle \tau\rangle)$.
Although these norms depend on the choice of a basis, two different 
choices lead to equivalent norms; in particular, the resulting
topologies are independant of any choice.
\end{cons}

\subsection{Logarithm and exponential}
\label{ssec:logexp}

To any Anderson $A$-motive $M$ over $R$, one can attach a logarithm
and an exponential map
$$\log_M : E_M\big(L[\![\tau]\!]\big) \longrightarrow \Lie_M\big(L[\![\tau]\!]\big),
\quad 
  \exp_M : \Lie_M\big(L[\![\tau]\!]\big) \longrightarrow E_M\big(L[\![\tau]\!]\big)$$
where $L[\![\tau]\!]$ is the ring of formal Ore series over $L$.
Although one could work out direct constructions of $\log_M$ and
$\exp_M$ in our setting, a shortcut consists in working with $M_L$
and its associated Anderson module $E$. Indeed, doing so, we
can simply define the logarithm (resp. exponential) of $M$ as
(the left-multiplication by) the logarithm (resp. exponential) of
$E$.

Those functions directly inherit the properties of the usual 
logarithm and exponential of Anderson modules. In particular, they
are $A \otimes R[\tau]$-linear and inverse one to the other.
They also exhibit important convergence properties.
We start by recalling the following classical result in the case 
of Anderson modules.

\begin{prop}
\label{prop:expphi}
Let $E$ be an Anderson module over $L$. Then
\begin{itemize}
\item[(i)] its exponential $\exp_E$ is in 
$M_d\big(L_\infty\langle\tau; \infty\rangle\big),$
\item[(ii)] there exists $\varepsilon > 0$ such that its logarithm
$\log_E$ is in $M_d\big(L_\infty\langle\tau; \varepsilon \rangle\big).$
\end{itemize}
\end{prop}

\begin{proof}
The first assertion is \cite[Proposition~2.14]{anderson} and the second follows from the non-archimedean inverse function theorem.
\end{proof}

In our setting, we have the following formulation.

\begin{prop}
\label{prop:expM}
The function $\exp_M$ induces by restriction and corestriction
an $A[\tau]$-linear map
$$\exp_M : \Lie_M\big(L_\infty\langle\tau\rangle\big) \longrightarrow 
  E_M\big(L_\infty\langle\tau\rangle\big)$$
which is uniformly locally an invertible isometry, \emph{i.e.}
there exists $\varepsilon > 0$ such that, for all
$x \in \Lie_M(L_\infty\langle\tau\rangle)$, $\exp_M$ induces
an invertible isometry between the disc of center $x$ and
radius $\varepsilon$ and the disc of center $\exp_M(x)$ and
radius $\varepsilon$.
\end{prop}

\begin{proof}
Let $E$ be the Anderson module over $L$ associated to $M_L$ as above.
After Proposition~\ref{prop:expphi}.(i),
the first assertion is a direct consequence of Lemma~\ref{lem:rdcproduct}.

Since $\exp_M$ is additive, it is enough to prove the second assertion for
$x = 0$. Writing $\exp_E = \sum_i e_i \tau^i$, if $x$ has small enough Gauss norm,
then $e_i \tau^i x$ has norm strictly smaller than $\Vert x \Vert$.
This ensures that $\Vert \exp_E x \Vert = \Vert x \Vert$ and so,
by additivity again, $\exp_M$ is locally an isometry.
The fact that it is locally invertible follows by combining
Lemma~\ref{lem:rdcproduct} and Proposition~\ref{prop:expphi}.(ii).
\end{proof}

\subsection{Euler factors and $L$-series}
\label{ssec:eulerfactor}

We fix an Anderson $A$-motive $M$ over $R$ together with a maximal
ideal $\p$ of $R$. We write $\Fp := R/\p$ for the residue field of
$R$ at $\p$ and we set $d := [\Fp:\FF]$.

Following Taelman~\cite{taelman} and
Angles--Ngo Dac--Tavares Ribeiro~\cite{ANT},
we would like to define the local factor of $R$ at $\p$ as the Fitting
ideal of $E_M(\Fp[\tau])$, viewed as a module over the commutative ring $A[\tau]$.
This definition should nevertheless be considered with caution because
it is not true in general that $\Fitt_{A[\tau]}(E_M(\Fp[\tau]))$ is a
principal ideal, while the local factor needs to be an element and not
an ideal.

In this article, we work around this issue by tensoring by $K$ over $A$.
We then consider the following ideal of the commutative ring $K[\tau]$:
$$Z_\p(M) := \Fitt_{K[\tau]}\big(K \otimes_A
E_M(\Fp[\tau])\big).$$

\begin{lem}
\label{lem:ZpFitt}
The ideal $Z_\p(M)$ is coprime with the ideal $(\tau)$.
\end{lem}

\begin{proof}
We first notice that $E_M(\Fp[\tau])$ is finitely generated over 
$\FF[\tau]$ (because $\Fp[\tau]$ is).
Hence $E_M(\Fp[\tau])$ is finitely generated and torsion over $A[\tau]$.
Besides, the multiplication by $\tau$ is injective on $E_M(\Fp[\tau])$
since it is injective on $\Fp[\tau]$ and the functor $E_M$ is exact.
\end{proof}

Given that $K[\tau]$ is a principal ideal domain, the ideal 
$Z_\p(M)$ is principal, and we let $z_\p(M) \in K[\tau]$ 
be one of its generator, which is defined up to multiplication by a 
nonzero element of $K$.
After Lemma~\ref{lem:ZpFitt}, we know moreover that $z_\p(M)$
is nonzero modulo $\tau$; we can then normalize it in such a way that
its constant coefficient is $1$.

\begin{prop}
\label{prop:eulertaun}
For all places $\p$,
the polynomial $z_\p(M)$ is a polynomial in $\tau^{\deg \p}$.
\end{prop}

\begin{proof}
This is well-known, \emph{e.g.} \cite[Proposition~13]{anderson-L}, but let us sketch the argument in our setting.
Let $\overline{K}$ denote an algebraic closure of $K$. Set $V:=K\otimes_A E_M(\Fp[\tau])$ and $\overline{V}:=\overline{K}\otimes_K V$.
After choosing an embedding $\Fp\hookrightarrow \overline{K}$, all the embeddings of $\Fp$ into $\overline{K}$ are given by $c\mapsto c^{q^i}$, for $i\in \{0,\ldots,d-1\}$, $d=\deg \p$.
We obtain decompositions
\begin{equation}\label{eq:decomposition-over-alg-clo}
\overline{K}\otimes \Fp =\prod_{i=0}^{d-1}{\overline{K}} \quad \text{and} \quad \overline{V} := \bigoplus_{i=0}^{d-1} V_i
\end{equation}
where $V_i$ is the $\overline{K}$-subspace of $\overline{V}$ on which $v\cdot c=(c^{q^i}\otimes_K 1) v$, for all $c\in \Fp$. 

From Lemma \ref{lem:ZpFitt} and its proof, $V$ is finite-dimensional over $K$ and the right action of $\tau$ is a linear automorphism. Over $\overline{K}$, the relation $\tau c=c^q\tau$ for $c\in \Fp$ implies that $V_{i+1}=V_i \tau$ and, by immediate induction, that $V_i=V_0\tau^i$. Since $V_0$ is a module over the commutative ring $\overline{K}[\tau^d]$, we obtain a canonical isomorphism
\[
\overline{V}=V_0\otimes_{\overline{K}[\tau^d]}\overline{K}[\tau].
\] 
Therefore, applying twice the base-change formula for Fitting ideals \href{https://stacks.math.columbia.edu/tag/07ZA}{[Stack Project: 07ZA(3)]}, we obtain 
\[
\Fitt_{K[\tau]}(V)=\Fitt_{\overline{K}[\tau]}(\overline{V})=\Fitt_{\overline{K}[\tau^d]}(V_0).
\]
We obtain $z_{\p}(M)\in K[\tau]\cap \overline{K}[\tau^d]=K[\tau^d]$. 
\end{proof}

\begin{prop}
\label{prop:eulerval}
There exists a constant $c \in (0,1)$ such that, for all place
$\p$, writing $z_\p(M) = 1 + z_1 \tau + \cdots + z_m \tau^m$,
we have $\vert z_i \vert \leq c^i$ for all $i$.
\end{prop}

\begin{proof}
Let $\CI$ be an algebraic closure of $K_\infty$.
Note that we have
\[
z_{\p}(M)=\prod_{\lambda}(1-\lambda^{-1}\tau)
\]
where the product is indexed over the eigenvalues of the $\CI$-linear right-action of $\tau$ on the finite dimensional $\CI$-vector space $\overline V_\infty:=\CI\otimes_A E_M(\Fp[\tau])$ counted with multiplicity. The proposition follows if one is able to show that $\vert \lambda \vert>d$ for some $d>1$ independent of $\p$ and the eigenvalue $\lambda$.

Choose a nonconstant element $a\in A$. We claim that there is an integer $s\geq 1$, dependent of $a$ but independent of $\p$, with the following property: 
one can choose a basis of the finite free right $\Fp[\tau]$-module $N_{\p} :=E_M(\Fp[\tau])$ in which the action of $a$ is represented by a matrix of $\tau$-degree at most $s$. 
Indeed, set $N:=E_M(R[\tau])=\Hom_{R[\tau]}(M,R[\tau])$ (\emph{i.e.} the comotive). Because $M$ is finite projective as a left $R[\tau]$-module, $N$ is finite projective as a right $R[\tau]$-module \cite[\S 2.Exercice 7]{lam}. In addition, by  \cite[\S 2.Exercice 20]{lam}, the formation of $N$ commutes with base-change; that is, for any $R$-algebra $S$, we have 
\[
N_S:=N\otimes_{R[\tau]}S[\tau] \cong \Hom_{S[\tau]}\big(M_S,S[\tau]\big).
\]
In particular, $N_{\Fp}\cong N_{\p}$ and $N_L$ is a free right $L[\tau]$-module. Choosing a basis of $N_L$ and inverting the finitely many denominator occuring in this isomorphism and its inverse, we find a nonzero $f\in R$ such that $N_{R[f^{-1}]}$ is free as a right $R[f^{-1}][\tau]$-module. There are only finitely many maximal ideals of $R$ containing $f$; at each of them $N_{\p}$ is free as $\Fp[\tau]$ is a (right) skew principal ideal domain. Choosing a basis at these finitely many places and increasing $s$ if necessary proves the claim.

Fix $\p$ and $d=\deg \p$. We shall show that $\vert \lambda\vert \geq \vert a \vert^{1/s}$, which finishes the proof. Fix a basis of $N_{\p}$ over $\Fp[\tau]$ as above; it induces a basis of $N_{\p}$ over $\mathbb{F}[\tau]$ by telescoping. The action of $a$ in this latter basis can be written as a polynomial in $\tau$ of degree at most $s$ and coefficients in matrices in $\mathbb{F}$.
Given $\lambda\in \CI$ an eigenvalue of $\tau$ and $v$ an associated eigenvector and comparing norms, we obtain 
\[
1<\vert a \vert \leq \max_{0\leq j\leq s} \vert \lambda \vert^{j},
\]
hence $\vert a \vert\leq \vert \lambda \vert^s$, as desired.
\end{proof}

\begin{deftn}[$L$-series]
Let $J$ be a nonzero ideal of $R$.
The \emph{$L$-series} of $M$ relative to $J$ is the function $L^{(J)}(M; \tau)$ defined by
$$L^{(J)}(M; \tau) = \prod_{\p \nmid J} z_\p(M)^{-1}$$
where the notation means that the product runs over all prime ideals $\p$ of $R$
not containing $J$.

\noindent
When $J = R$, we simply write $L(M; \tau)$ for $L^{(R)}(M; \tau)$.
\end{deftn}

It follows from Propositions~\ref{prop:eulertaun} and~\ref{prop:eulerval}
that $L^{(J)}(M; \tau)$ converges in both rings $K[\![\tau]\!]$ and
$$K_\infty\langle\tau\rangle := \Big\{ \sum_i a_i \tau^i \,\big|\, a_i \in K_\infty, \, 
\lim_{i \to \infty} \vert a_i \vert = 0 \Big\}.$$
Indeed, the fact that $z_\p(M)$ is a polynomial in $\tau^{\deg \p}$
implies the formal convergence, while the estimation of the $\infty$-adic
valuations shows the convergence in $K_\infty\langle\tau\rangle$.

\section{The module of units and the class module}
\label{sec:units}
Following the work of Anglès--Ngo Dac--Tavares Ribeiro \cite{ANT}, and building on foundational ideas of Taelman \cite{taelman}, we introduce the relevant \emph{module of units}.
\begin{deftn}[Module of units]
Let $J$ be a nonzero ideal of $R$.
The \emph{module of $J$-units of $M$} is:
$$U_M^{(J)} = \big\{ x \in \Lie_M\big(L_\infty\langle \tau\rangle\big)
\,\big|\,
\exp_M(x) \in E_M\big(J[\tau]\big) \big\}.$$

\noindent
When $J = R$, we simply write $U_M$ for $U_M^{(R)}$.
\end{deftn}

In other words, $U_M^{(J)}$ is the kernel of the map
$$e_M^{(J)} : \Lie_M(L_\infty\langle \tau\rangle) \longrightarrow
\frac{E_M\big(L_\infty\langle \tau\rangle\big)}{E_M\big(J[\tau]\big)} \simeq
E_M\big(L_\infty/J\langle \tau\rangle\big)$$
induced by the exponential.
Here, we note that the last isomorphism follows from the exactness of $E_M$.
The cokernel of $e_M^{(J)}$ will also play a central role; by definition,
it is the \emph{class module of $M$ relative to $J$} and it will be
denoted by $H_M^{(J)}$ (or just $H_M$ when $J = R$).
We therefore have the following fundamental exact sequence:
$$0 \longrightarrow
U_M^{(J)} \longrightarrow 
\Lie_M\big(L_\infty\langle \tau\rangle\big) \stackrel{e_M^{(J)}}\longrightarrow
E_M\big(L_\infty/J\langle \tau\rangle\big) \longrightarrow
H_M^{(J)} \longrightarrow 0.$$
Noticing that $e_M^{(J)}$ is $A[\tau]$-linear, we conclude that 
$U_M^{(J)}$ and $H_M^{(J)}$ are both $A[\tau]$-modules.

\subsection{Finiteness of the class module}
If $Q$ is a quotient of $E_M(L_\infty\langle \tau\rangle)$
by an $\FF[\tau]$-submodule, we endow it with the quotient seminorm
defined by
$$\Vert q \Vert = \inf_{x \in \text{pr}^{-1}(q)} \Vert x \Vert$$
where $\text{pr} : E_M(L_\infty\langle \tau\rangle) \to Q$
is the canonical projection. We underline that $\Vert \cdot \Vert$
on $Q$ is not necessarily a norm: in full generality, we can have
$\Vert q \Vert$ = 0 for some nonzero $q \in Q$.

\begin{lem}
\label{lem:fgtopology}
Let $Q$ be a quotient of $E_M(L_\infty\langle \tau\rangle)$
by a $\FF[\tau]$-submodule. We assume that:
\begin{itemize}
\item $Q$ is bounded, \emph{i.e.} there exists a constant $R > 0$
such that $\Vert q \Vert \leq R$ for all $q \in Q$,
\item $Q$ is discrete, \emph{i.e.} there exists a constant $r > 0$
such that $\Vert q \Vert \geq r$ for all $q \in Q$, $q \neq 0$.
\end{itemize}
Then $Q$ is finitely generated over $\FF[\tau]$.
\end{lem}

\begin{proof}
We choose a bicontinuous $\FF[\tau]$-linear isomorphism 
$E_M(L_\infty\langle \tau\rangle) \simeq L_\infty\langle \tau\rangle^d$
where the codomain is equipped with the standard sup norm.
Let $B_\rho \subset L_\infty^d$ denote the closed ball of radius~$\rho$.
Then $B_\rho\langle\tau\rangle$ is the closed ball of radius~$\rho$
in $L_\infty\langle\tau\rangle^d$. Besides, we observe that $B_\rho$
is a $\FF$-vector space and $B_\rho\langle\tau\rangle$ is a right
$\FF[\tau]$-module.
By discreteness of $Q$, there exists $r' > 0$ such that the ball $B_{r'}\langle \tau \rangle$
maps to $0$ in $Q$.
Therefore, for all $R' > r'$, there is a $\FF[\tau]$-linear map 
$(B_{R'}\langle\tau\rangle/ B_{r'}\langle\tau\rangle)^d \to Q$. 
By the boundedness assumption, this map is surjective for $R'$ large 
enough. Hence it is enough to prove that 
$(B_{R'}\langle\tau\rangle/B_{r'}\langle\tau\rangle)^d$ is finitely 
generated as a $\FF[\tau]$-module.
But, we have an identification
$B_{R'}\langle\tau\rangle/B_{r'}\langle\tau\rangle \simeq (B_{R'}/B_{r'})[\tau]$
and we know that the quotient $B_{R'}/B_{r'}$ has finite dimension over
$\FF$. The lemma follows.
\end{proof}

We apply Lemma~\ref{lem:fgtopology} with $Q = H_M^{(J)}$, which
is the quotient of $E_M(L_\infty\langle \tau\rangle)$ by the
submodule
$E_M(J[\tau]) + \exp_M(\Lie_M(L_\infty\langle \tau\rangle))$.
It follows from the fact that $\exp_M$ is uniformly locally an 
invertible isometry
(see Proposition~\ref{prop:expM}) that $H_M^{(J)}$ is discrete. On the
other hand, $H_M^{(J)}$ is bounded being a quotient of
$E_M(L_\infty/J \langle \tau\rangle) \simeq 
(L_\infty/J \langle \tau\rangle)^d$, which is itself
bounded.
We conclude that $H_M^{(J)}$ is finitely generated over $\FF[\tau]$.

The following generalizes \cite[Proposition~2.2]{ANT}.
\begin{prop}
\label{prop:HMJ}
The module $H_M^{(J)}$ is finite and the multiplication by $\tau$ 
acts bijectively on it. Moreover, we have an exact sequence
$$0 \longrightarrow 
  U_M^{(J)} \stackrel{\times\tau}{\longrightarrow}
  U_M^{(J)} \longrightarrow
  \Lie_M(J) \longrightarrow 0.$$
\end{prop}

\begin{proof}
We consider the following diagram with exact rows:
$$\begin{tikzcd}
  0 \ar[r]
& \Lie_M\big(L_\infty\langle\tau\rangle\big) \ar[r, "\times\tau"] \ar[d,"\exp_M"]
& \Lie_M\big(L_\infty\langle\tau\rangle\big) \ar[r] \ar[d,"\exp_M"]
& \Lie_M\big(L_\infty\big) \ar[r] \ar[d,"\alpha"]
& 0 \\
  0 \ar[r]
& E_M\big(L_\infty/J\langle\tau\rangle\big) \ar[r, "\times\tau"]
& E_M\big(L_\infty/J\langle\tau\rangle\big) \ar[r]
& E_M\big(L_\infty/J\big) \ar[r]
& 0
\end{tikzcd}$$
where $\alpha$ is the induced map on cokernels.
We underline that, in the above diagram, $L_\infty/J$ is identified 
with the cokernel of the multiplication by $\tau$ on 
$L_\infty/J\langle\tau\rangle$. In particular, $\tau$ acts by zero
on it and so, we have $E_M(L_\infty/J) = \Lie_M(L_\infty/J)$.
Using now the fact that
$\exp_M$ is the identity modulo~$\tau$, we conclude that $\alpha$
is nothing but the map induced by the canonical projection
$L_\infty \to L_\infty/J$. Thus $\alpha$ is surjective and its
kernel is $\Lie_M(J)$.

Applying the snake lemma, we then get the long exact sequence
$$0 \longrightarrow 
  U_M^{(J)} \stackrel{\times\tau}{\longrightarrow}
  U_M^{(J)} \longrightarrow
  \Lie_M(J) \longrightarrow
  H_M^{(J)} \stackrel{\times\tau}{\longrightarrow}
  H_M^{(J)} \longrightarrow 0.$$
In particular, the multiplication by $\tau$ acts surjectively
on $H_M^{(J)}$.
Since the latter is finitely generated over $\FF[\tau]$, we know
moreover that it is isomorphic to a direct sum of copies of 
$\FF[\tau]$ or finite quotients of it.
Since the multiplication by $\tau$ on $\FF[\tau]$ is obviously
nonsurjective, we conclude that $H_M^{(J)}$ cannot have any 
summand of this form. Therefore, it is finite, from what we
deduce that $\tau$ acts bijectively on it. The exactness of 
the sequence
$$0 \longrightarrow 
  U_M^{(J)} \stackrel{\times\tau}{\longrightarrow}
  U_M^{(J)} \longrightarrow
  \Lie_M(J) \longrightarrow 0$$
finally follows.
\end{proof}

\subsection{Projectivity of the module of units}
In this subsection, we prove that modules of units are projective over $A[\tau]$. It generalizes \cite[Proposition~1]{AT} which shows this after localization with respect to $S=\mathbb{F}[\tau]\setminus \{0\}$. This is specifically harder in our situation as $A[\tau]$ has Krull dimension $2$. 

We start by recalling classical facts about the ambient space
$\Lie_M(L_\infty\langle\tau\rangle)$ in which $U_M^{(J)}$
lives. First of all, we notice that
$$\Lie_M(L_\infty\langle\tau\rangle) \simeq
\Lie_M(L_\infty)\langle\tau\rangle$$
where the latter is defined as the set of series $\sum_i a_i \tau^i$
with $a_i \in \Lie_M(L_\infty)$ and $\lim_{i \to \infty} \vert a_i 
\vert = 0$.
Moreover, it is well known (see, \emph{e.g.}, \cite[Lemma 3.5]{ferraro}) that the action of $A$
on $\Lie_M(L_\infty)$ uniquely extends to a continuous action of $K_\infty$
turning $\Lie_M(L_\infty)$ into a finite dimensional $K_\infty$-vector
space. We deduce that $\Lie_M(L_\infty\langle\tau\rangle)$
is a free module over $K_\infty\langle\tau\rangle$.

The aim of this subsection is to prove the following theorem.

\begin{thm}
\label{thm:UMJ}
The $A[\tau]$-module $U_M^{(J)}$ is projective of finite rank.
Moreover, the canonical map
$$\iota_M^{(J)} :
 K_\infty\langle\tau\rangle \otimes_{A[\tau]} U_M^{(J)} 
 \longrightarrow \Lie_M\big(L_\infty\langle\tau\rangle\big)$$
is an isomorphism.
\end{thm}

We start by showing its finiteness.
\begin{lem}
\label{lem:UMJfg}
The $A[\tau]$-module $U_M^{(J)}$ is finitely generated.
\end{lem}

\begin{proof}
Let $\mathcal U$ be the $K_\infty\langle\tau\rangle$-span of $U_M^{(J)}$
in $\Lie_M(L_\infty\langle\tau\rangle)$.
From the fact that $K_\infty\langle\tau\rangle$ is a noetherian ring, 
we deduce that $\mathcal U$ is finitely generated over $K_\infty\langle\tau\rangle$.
Let $x_1, \ldots, x_m \in U_M^{(J)}$ be such a generating family and
let $U$ be their $A[\tau]$-span.

We claim that, if $\varepsilon$ is small
enough, there is no nonzero $x \in U_M^{(J)}$ with $\Vert x \Vert \leq
\varepsilon$. Indeed, if such an $x$ existed, we would deduce that
$\exp_M(x) \in E_M(J[\tau])$ and $\Vert \exp_M(x) \Vert \leq 
\varepsilon$ by Proposition~\ref{prop:expM}; this is a contradiction
when $\varepsilon$ is small.

The space $U_M^{(J)}$ is then discrete and so is the quotient $U_M^{(J)}/U$.
Moreover, the latter is bounded since $K_\infty\langle\tau\rangle / A[\tau]
= K_\infty/A \langle\tau\rangle$ is also.
By an analogue of Lemma~\ref{lem:fgtopology}, we 
conclude that $U_M^{(J)} / U$ is finitely generated over $\FF[\tau]$ 
and hence \emph{a fortiori} over $A[\tau]$. The lemma follows.
\end{proof}

We continue with a general result about flatness in commutative algebra.

\begin{prop}
\label{prop:flatness}
Let $L$ be a flat $A[\tau]$-module and let $U$ be a $A[\tau]$-submodule
such that $L/U$ has no $\FF[\tau]$-torsion. Then $U$ is $A[\tau]$-flat.
\end{prop}

\begin{proof}
We follow the strategy used in \cite[Proposition 4.33]{gazda}.
Let $I\subset A[\tau]$ be a nonzero ideal. By \href{https://stacks.math.columbia.edu/tag/00HD}{[Stack~Project:~00HD]},
it is enough to show that the multiplication map $\mu:I\otimes_{A[\tau]} U\to U$ is injective. The latter fits into a commutative square
\[
\begin{tikzcd}
I\otimes_{A[\tau]} U \arrow[r,"\mu"]\arrow[d,"\rho"] & U\arrow[d] \\
I\otimes_{A[\tau]} L \arrow[r] & L
\end{tikzcd}
\]
where the lower horizontal map and the right-most vertical map are injective because $L$ is flat over $A[\tau]$.
In particular, $\mu$ is injective if, and only if $\rho$ is.
To show that $\rho$ is injective, it is enough to show that the $A[\tau]$-module $T:=\mathrm{Tor}^{A[\tau]}_1(I,L/U)$ vanishes (see \href{https://stacks.math.columbia.edu/tag/00M0}{[Stack~Project:~00M0]}). Write $S=\mathbb{F}[\tau]\setminus \{0\}$ as above; we show it by proving that
\begin{enumerate}
\item\label{item:torsionfree} $T$ is $\FF[\tau]$-torsion free,
\item\label{item:torsion} that $T[S^{-1}]=0$.
\end{enumerate} 
To prove \eqref{item:torsionfree}, let $a\in S$ and define $C$ as the cokernel of the multiplication by $a$ on $L/U$.
Since $L/U$ has no $\FF[\tau]$-torsion, the sequence $0\to L/U\stackrel{\times a}{\to} L/U\to C\to 0$ is exact.
Hence, tensoring it by $I$ produces a long exact sequence which includes
\[
\mathrm{Tor}_{2}^{A[\tau]}(I,C)\longrightarrow T\xrightarrow{\times a} T.
\]
To show that the $\mathrm{Tor}_{2}$ vanishes, we tensor the sequence $0\to I\to A[\tau]\to A[\tau]/I\to 0$ along $C$, obtaining the exact sequence
\[
\mathrm{Tor}_{3}^{A[\tau]}(A[\tau]/I,C)\longrightarrow \mathrm{Tor}_{2}^{A[\tau]}(I,C) \longrightarrow \mathrm{Tor}_{2}^{A[\tau]}(A[\tau],C).
\]
Both extremal terms vanish, the former because $A[\tau]$ is a regular noetherian domain of Krull dimension $2$, and the latter because $A[\tau]$ is flat over itself. Hence $T$ is $a$-torsion free.\\
To prove \eqref{item:torsion}, we use that the $\mathrm{Tor}$-functor commutes with localization in both variables:
\[
S^{-1}T\cong \mathrm{Tor}^{A[\tau][S^{-1}]}_1(I\otimes_{\mathbb{F}[\tau]}\mathbb{F}(\tau),(L/U)\otimes_{\mathbb{F}[\tau]}\mathbb{F}(\tau)).
\]
Now $I\otimes_{\mathbb{F}[\tau]}\mathbb{F}(\tau)$ is an ideal in the Dedekind domain $A[\tau][S^{-1}]$, hence is flat, and thus the functor $\mathrm{Tor}^{A[\tau][S^{-1}]}_1(I\otimes_{\mathbb{F}[\tau]}\mathbb{F}(\tau),-)$ is zero. Thus $S^{-1}T$ vanishes, as desired.
\end{proof}

\begin{proof}[Proof of Theorem~\ref{thm:UMJ}]
We recall that $\Lie_M(L_\infty\langle\tau\rangle)$ is free over $K_\infty\langle\tau\rangle$. Hence, it is flat over this ring.
Besides, the maps $A[\tau]\to K[\tau]$ and $K[\tau]\to K_{\infty}\langle \tau \rangle$ are both flat, the former being a localization and the latter being the $\infty$-adic completion of a noetherian domain.
Thus $K_{\infty}\langle \tau \rangle$ is flat over $A[\tau]$ and we conclude that $\Lie_M(L_\infty\langle\tau\rangle)$ is $A[\tau]$-flat as well.

On the other hand, we observe that the quotient $\Lie_M(L_\infty\langle\tau\rangle) / U_M^{(J)}$ embeds into $E_M(L_\infty/J\langle\tau\rangle)$.
Since the latter has no $\FF[\tau]$-torsion, the same holds for the former.
We can then apply Proposition~\ref{prop:flatness} with $L = \Lie_M(L_\infty\langle\tau\rangle)$ and $U = U_M^{(J)}$ and deduce that
$U_M^{(J)}$ is $A[\tau]$-flat. Since it is also finitely generated (Lemma~\ref{lem:UMJfg}),
we conclude that it is projective by \href{https://stacks.math.columbia.edu/tag/00NX}{[Stack~Project:~00NX]}.

It remains to prove that the map $\iota_M^{(J)}$ is an isomorphism.
After Proposition~\ref{prop:HMJ}, we know that
$$K_\infty \otimes_{A[\tau]} U_M^{(J)} 
\simeq K_\infty \otimes_A \Lie_M(J) 
\simeq \Lie_M(L_\infty)
\simeq K_\infty \otimes_{K_\infty\langle\tau\rangle} 
  \Lie_M\big(L_\infty\langle\tau\rangle\big)$$
where, in the leftmost and rightmost tensor products, the base change
map is induced by $\tau \mapsto 0$.
We derive that the domain and the codomain of $\iota_M^{(J)}$ have the same rank.
Therefore, it is enough to prove that $\iota_M^{(J)}$ is surjective.
For this, we notice that the cokernel of $\iota_M^{(J)}$ is naturally a quotient
of $\Lie_M(L_\infty\langle\tau\rangle) / U_M^{(J)}$ which itself embeds
into $E_M(L_\infty/J\langle\tau\rangle)$. Since the latter is bounded, 
so is $\coker \iota_M^{(J)}$.
Given that $\coker \iota_M^{(J)}$ is a $K_\infty$-vector space, this can only
happen if it vanishes, that is, if $\iota_M^{(J)}$ is surjective.
\end{proof}

\section{The class formula}
\label{sec:classformula}

By Theorem~\ref{thm:UMJ}, $U_M^{(J)}$ appears as an $A[\tau]$-\emph{lattice}
inside $\Lie_M(L_\infty\langle\tau\rangle)$.
The aim of the class formula is to compare it with the other lattice $\Lie_M(J[\tau])$.
The comparison holds at the level of determinants. Let $d$ be the \emph{dimension} of $M$,
that is by definition the $A$-rank of $\Lie_M(R)$ or, equivalently, the
$K$-dimension of $\Lie_M(L)$. In what follows, the notation $\det$ refers
to the $d$-th exterior power.

\begin{thm}
\label{thm:classformula}
For all nonzero ideal $J$ of $R$, we have an equality of $A[\tau]$-lines
$$\mathrm{det}_{A[\tau]} \big(U_M^{(J)}\big) = 
L^{(J)}(M; \tau) \cdot \mathrm{det}_{A[\tau]} \big(\Lie_M(J[\tau])\big)$$
inside $\det_{K_\infty\langle\tau\rangle} \big(\Lie_M(L_\infty\langle\tau\rangle)\big)$.
\end{thm}

\subsection{Proof of Theorem~\ref{thm:classformula}}

The first ingredient is the famous Quillen--Suslin theorem~\cite{quillen}
which, in our case of interest, has the following consequence: every finitely 
generated projective module over $A[\tau]$ comes by scalar extension from a unique
finitely generated projective module over $A$. In particular, two such modules $U$
and $V$ are isomorphic if and only if $U/U \tau$ and $V/V \tau$ are isomorphic as
$A$-modules.

In our situation, we know from Proposition~\ref{prop:HMJ} that
$U_M^{(J)}/U_M^{(J)}\tau \simeq \Lie_M(J)$. Hence we conclude that
$U_M^{(J)}$ and $\Lie_M(J[\tau])$ are abstractly isomorphic as $A[\tau]$-module.
Let $\alpha^{(J)} : \Lie_M(J[\tau]) \to U_M^{(J)}$ be such an isomorphism,
chosen in such a way that $\alpha^{(J)} \bmod \tau$ is the identity.
Extending scalars to $K_\infty\langle\tau\rangle$, we find that $\alpha^{(J)}$
induces an automorphism of $\Lie_M(L_\infty\langle\tau\rangle)$.
Writing $\lambda^{(J)}$ for its determinant, we have 
$\lambda^{(J)} \in K_\infty\langle\tau\rangle^\times$,
$\lambda^{(J)} \equiv 1 \pmod \tau$ and
$$\mathrm{det}_{A[\tau]} \big(U_M^{(J)}\big) = 
\lambda^{(J)} \cdot \mathrm{det}_{A[\tau]} \big(\Lie_M(J[\tau])\big).$$
Again, when $J = R$, we simply write $\lambda$ for $\lambda^{(R)}$.

\begin{lem}
\label{lem:UMexact}
For all nonzero ideal $J$ of $R$, we have an exact sequence
$$0 \longrightarrow 
  K \otimes_A U_M^{(J)} \longrightarrow
  K \otimes_A U_M \longrightarrow
  K \otimes_A E_M\big(R/J[\tau]\big) \longrightarrow 0.$$
\end{lem}

\begin{proof}
We apply the snake lemma to the following diagram with exact rows:
$$\begin{tikzcd}
  0 \ar[r]
& 0 \ar[r] \ar[d] 
& \Lie_M\big(L_\infty\langle\tau\rangle\big) \ar[r, "\id"] \ar[d,"\exp_M"]
& \Lie_M\big(L_\infty\langle\tau\rangle\big) \ar[r] \ar[d,"\exp_M"]
& 0 \\
  0 \ar[r]
& E_M\big(R/J[\tau]\big) \ar[r]
& E_M\big(L_\infty/J\langle\tau\rangle\big) \ar[r]
& E_M\big(L_\infty/R\langle\tau\rangle\big) \ar[r]
& 0
\end{tikzcd}$$
and get the long exact sequence
$$0 \longrightarrow 
  U_M^{(J)} \longrightarrow
  U_M \longrightarrow
  E_M\big(R/J[\tau]\big) \longrightarrow 
  H_M^{(J)} \longrightarrow
  H_M \longrightarrow 0.$$
The lemma follows after tensoring by $K$ (which is flat over $A$) and
remembering that $H_M^{(J)}$ is finite, implying that $K \otimes_A
H_M^{(J)} = 0$.
\end{proof}

\begin{lem}
\label{lem:lambdaJ}
For all nonzero ideal $J$ of $R$, we have the equality
$$\frac{\lambda^{(J)}}{\lambda} = \frac{L^{(J)}(M;\tau)}{L(M;\tau)}.$$
\end{lem}

\begin{proof}
Taking determinants in the exact sequence of Lemma~\ref{lem:UMexact},
we obtain
\begin{equation}
\label{eq:lambdaJ:1}
\lambda^{(J)} = \lambda \cdot 
\Fitt_{K[\tau]}\big(K \otimes_A E_M(R/J[\tau])\big).
\end{equation}
Let $J = \p_1^{n_1} \cdots \p_s^{n_s}$ be the decomposition of $J$ in
a product of primes. Here, the $\p_i$ are then pairwise distinct prime
ideals of $R$ and the $n_i$ are positive integers. 
By the chinese remainder theorem, $R/J$ is isomorphic to the product
of the $R/\p_i^{n_i}$ and thus
\begin{equation}
\label{eq:lambdaJ:2}
E_M\big(R/J[\tau]\big) \simeq
  E_M\big(R/\p_1^{n_1}[\tau]\big) \oplus
  E_M\big(R/\p_2^{n_2}[\tau]\big) \oplus \cdots \oplus
  E_M\big(R/\p_s^{n_s}[\tau]\big)
\end{equation}
as $A[\tau]$-modules.
We now claim that, for any prime ideal $\p$ and any integer $n \geq 1$,
the canonical projection $R/\p^{n+1} \to R/\p^n$ induces an isomorphism
$$K \otimes_A E_M\big(R/\p^{n+1}[\tau]\big) \longrightarrow
  K \otimes_A E_M\big(R/\p^n[\tau]\big).$$
Indeed, by exactness of $E_M$ and flatness of $K$ over $A$, the above map
is surjective and its kernel is 
$K \otimes_A E_M(\p^n/\p^{n+1}[\tau])$.
Observing that the left multiplication by $\tau$ is $0$ on
$\p^n/\p^{n+1}[\tau]$, we deduce further that
$$K \otimes_A E_M\big(\p^n/\p^{n+1}[\tau]\big) \simeq
  K \otimes_A \Lie_M\big(\p^n/\p^{n+1}[\tau]\big) \simeq
  K \otimes_A \Lie_M\big(\p^n/\p^{n+1}\big)[\tau]$$
and the latter vanishes given that $\Lie_M(\p^n/\p^{n+1})$
is killed by a power of $\p$.
This proves the claim, which in turn implies by induction that
\begin{equation}
\label{eq:lambdaJ:3}
K \otimes_A E_M\big(R/\p_i^{n_i}[\tau]\big) \simeq 
  K \otimes_A E_M\big(R/\p_i[\tau]\big)
\end{equation}
for all $i$. Putting together Equations~\eqref{eq:lambdaJ:1},
\eqref{eq:lambdaJ:2} and~\eqref{eq:lambdaJ:3}, we obtain
$$\lambda^{(J)} = \lambda \cdot
  \prod_{i=1}^s \Fitt_{K[\tau]}\big(K \otimes_A E_M\big(R/\p_i[\tau]\big)\big).$$
The lemma follows after remembering the definition of the $L$-series.
\end{proof}

\begin{lem}
\label{lem:Jn}
For all positive integer $n$, there exists an ideal $J_n$ of $R$
such that
$$U_M^{(J)} \subset 
  \Lie_M\big(J[\tau]\big) + 
  \Lie_M\big(L_\infty\langle\tau\rangle\big) \tau^n$$
for all $J \subset J_n$.
\end{lem}

\begin{proof}
After choosing coordinates, the exponential and logarithm maps are given
by the left multiplication by series of the form
$$1 + e_1 \tau + e_2 \tau^2 + \cdots \in M_d(L)[\![\tau]\!] 
\quad \text{and} \quad
1 + \ell_1 \tau + \ell_2 \tau^2 + \cdots \in M_d(L)[\![\tau]\!]$$
respectively.
We define $J_n$ as the ideal generated by a common denominator of the
entries of $e_1, \ldots, e_n, \ell_1, \ldots, \ell_n$. For $J \subset J_n$, 
the map $\exp_M$ takes $\Lie_M(J[\tau]/\tau^n)$
to $E_M(J[\tau]/\tau^n)$ and, conversely, $\log_M$ takes
$E_M(J[\tau]/\tau^n)$ to $\Lie_M(J[\tau]/\tau^n)$. It
follows that $\exp_M$ induces a bijection between
$\Lie_M(J[\tau]/\tau^n)$ and $E_M(J[\tau]/\tau^n)$.

We now consider the following diagram with exact rows:
$$\begin{tikzcd}
  0 \ar[r]
& \Lie_M\big(L_\infty\langle\tau\rangle\big) \ar[r, "\times\tau^n"] \ar[d,"\exp_M"]
& \Lie_M\big(L_\infty\langle\tau\rangle\big) \ar[r] \ar[d,"\exp_M"]
& \Lie_M\big(L_\infty[\tau]/\tau^n\big) \ar[r] \ar[d,"\alpha"]
& 0 \\
  0 \ar[r]
& E_M\big(L_\infty/J\langle\tau\rangle\big) \ar[r, "\times\tau^n"]
& E_M\big(L_\infty/J\langle\tau\rangle\big) \ar[r]
& E_M\big((L_\infty/J)[\tau]/\tau^n\big) \ar[r]
& 0
\end{tikzcd}$$
where $\alpha$ is the induced map on cokernels. By the first part of the
proof, the kernel of $\alpha$ is $\Lie_M(J[\tau]/\tau^n)$. Therefore,
by the snake lemma, we have an exact sequence
$$0 \longrightarrow 
  U_M^{(J)} \stackrel{\tau^n}\longrightarrow
  U_M^{(J)} \longrightarrow
  \Lie_M\big(J[\tau]/\tau^n\big) \longrightarrow
  H_M^{(J)}.$$
Hence $U_M^{(J)} \subset \Lie_M(J[\tau]) + U_M^{(J)} \tau^n$,
which implies the lemma.
\end{proof}

\begin{proof}[Proof of Theorem~\ref{thm:classformula}]
After Lemma~\ref{lem:lambdaJ}, it is enough to prove Theorem~\ref{thm:classformula} for $J = R$.
We set $\mu := \lambda \cdot {L(M;\tau)}^{-1} \in K_\infty\langle\tau\rangle^\times$
and aim at proving that $\mu = 1$.
Let $n$ be a positive integer, and let $J$ be an ideal included in the ideal
$J_n$ of Lemma~\ref{lem:Jn}.
From Lemma~\ref{lem:Jn}, we deduce that
$$\mathrm{det}_{A[\tau]} \big(U_M^{(J)}\big) \subset 
  \mathrm{det}_{A[\tau]} \big(\Lie_M(J[\tau])\big) + 
  \mathrm{det}_{K_\infty\langle\tau\rangle} \big(\Lie_M(L_\infty\langle\tau\rangle)\big) \tau^n.$$
Hence $\lambda^{(J)} \in A[\tau] + K_\infty\langle\tau\rangle \tau^n$ which,
after Lemma~\ref{lem:lambdaJ}, can be rewritten as
$$\mu \cdot L^{(J)}(M;\tau) \in A[\tau] + K_\infty\langle\tau\rangle \tau^n.$$
If we now assume in addition that $J$ is contained in all ideals of $R$ of
degree at most $n$, we further know by Proposition~\ref{prop:eulertaun}
that $L^{(J)}(M;\tau) \equiv 1 \pmod{\tau^n}$ and therefore simply get
$\mu \in A[\tau] + K_\infty\langle\tau\rangle \tau^n$.

Since this conclusion holds for any $n$, we obtain $\mu \in A[\tau]$.
The fact that $\mu$ is invertible in $K_\infty\langle\tau\rangle$ and that
it is congruent to $1$ modulo $\tau$ finally ensures that $\mu = 1$ as desired.
\end{proof}

\subsection{A nondeterminantal version of the class formula}
\label{ssec:nondet}

Examining the proof of Theorem~\ref{thm:classformula}, we see that many
intermediate results hold before taking determinants: it is the case for
Lemma~\ref{lem:UMexact}, Lemma~\ref{lem:Jn} and the key 
equations~\eqref{eq:lambdaJ:2} and  \eqref{eq:lambdaJ:3} involved in
the proof of Lemma~\ref{lem:lambdaJ}.
One may then wonder if a ``nondeterminantal'' class formula, comparing
directly $U_M$ and $\Lie_M(R)$, could exist.

However, writing down a precise statement in this direction is not
obvious since finding a ``nondeterminantal'' version of the $L$-series
looks like a difficult task. Nonetheless, after Lemma~\ref{lem:lambdaJ},
we observe that the class formula is equivalent to
$$\lim_{J \to 0} \lambda^{(J)} = 1 
\qquad\text{($\infty$-adic convergence),}$$
meaning that, for all $\varepsilon > 0$,
there exists a nonzero ideal $J_\varepsilon \subset R$ such that
$\Vert \lambda^{(J)} -1 \Vert \leq \varepsilon$ for all nonzero ideal
$J \subset J_\varepsilon$.
Besides $\lambda^{(J)}$, which was previously defined as an element
of $K_\infty\langle\tau\rangle^\times$, can alternatively be thought
as an isomorphism between $\det_{A[\tau]}(\Lie_M(J[\tau]))$
and $\det_{A[\tau]}(U_M^{(J)})$. Given that $\Lie_M(J[\tau])$
and $U_M^{(J)}$ are already isomorphic as modules, a
``nondeterminantal'' version of $\lambda^{(J)}$ also makes sense:
it is the choice of an isomorphism $\alpha^{(J)} : \Lie_M(J[\tau])
\stackrel{\sim}{\longrightarrow} U_M^{(J)}$.

In order to compare the $\alpha^{(J)}$ between them, we need to
embed them in a common space. For this, we just extend scalars to
$K_\infty\langle\tau\rangle$, which allows to view $\alpha^{(J)}$
as a $K_\infty\langle\tau\rangle$-linear endomorphism of
$\Lie_M(L_\infty\langle\tau\rangle)$, that is, an element of
$$\mathcal E := 
\End_{K_\infty\langle\tau\rangle}\big(\Lie_M(L_\infty\langle\tau\rangle)\big).$$
Finally, we introduce topology by endowing $\mathcal E$ with the
classical endomorphism norm, namely
$$\Vert f \Vert = \sup_{\Vert x \Vert \leq 1} \Vert f(x) \Vert$$
for a choice of norm on $\Lie_M(L_\infty\langle\tau\rangle)$ as
introduced in Subsection~\ref{ssec:logexp}. We emphasize that the
resulting topology does not depend on any choice.

In this setting, a ``nondeterminantal'' version of the class
formula can be formulated as follows.

\begin{conj}
\label{conj:nondet:classformula}
Let $\mathcal J$ be the set of nonzero ideals of $R$.
There exists a sequence $(\alpha^{(J)})_{J \in \mathcal J}$
of elements of $\mathcal E$ such that:
\begin{itemize}
\item for all $J \in \mathcal J$, we have
$\alpha^{(J)}\big(\Lie_M(J[\tau])\big) = U_M^{(J)}$,
\item $\lim_{J \to 0} \alpha^{(J)} = \id_{\Lie_M(L_\infty\langle\tau\rangle)}$.
\end{itemize}
\end{conj}

Conjecture~\ref{conj:nondet:classformula}
admits a positive
answer if we replace the $\infty$-adic topology by of the $\tau$-adic
topology, as a consequence of Lemma~\ref{lem:Jn}. The difficulty is then 
really to switch between topologies. In the one-dimensional case, which
occurs after taking determinants, we bypassed this difficulty by
using the equality
$$\big(1 + K_\infty\langle\tau\rangle \tau\big) 
  \cap K_\infty\langle\tau\rangle^\times \cap A[\tau] = \{1\}.$$
In higher dimension, similar techniques can be used to prove that
Conjecture~\ref{conj:nondet:classformula} indeed holds true after
extending scalars from $A$ to $K$.
Precisely, we have the following theorem.

\begin{thm}
\label{thm:nondet:classformula}
There exists a sequence $(\alpha^{(J)})_{J \in \mathcal J}$
of elements of $\mathcal E$ such that:
\begin{itemize}
\item for all $J \in \mathcal J$, we have
$\alpha^{(J)}\big(\Lie_M(L[\tau])\big) = K \otimes_A U_M^{(J)}$,
\item $\lim_{J \to 0} \alpha^{(J)} = \id_{\Lie_M(L_\infty\langle\tau\rangle)}$.
\end{itemize}
\end{thm}

Before proving the theorem, we recall the following proposition.

\begin{prop}
\label{prop:SLdense}
$\mathrm{SL}_d(K[\tau])$ is dense in $\mathrm{SL}_d(K_\infty\langle\tau\rangle)$.
\end{prop}

\begin{proof}
Since $K_\infty\langle\tau\rangle$ is an Euclidean ring,
the group $\mathrm{SL}_d(K_\infty\langle\tau\rangle)$ is generated
by its transvections.
Given that any such transvection can be realized as a limit of transvections
over $K[\tau]$, the proposition follows.
\end{proof}

\begin{proof}[Proof of Theorem~\ref{thm:nondet:classformula}]
Throughout the proof, we fix a $K[\tau]$-basis $\mathcal B$ of
$\Lie(L[\tau])$.

Let $J$ be nonzero ideal of $A$. We consider an isomorphism
$\beta : \Lie_M(J[\tau]) \to U_M^{(J)}$ and we let $B \in
\mathrm{GL}_d(K_\infty\langle\tau\rangle)$ denote its
matrix in the distinguished basis $\mathcal B$.
We normalize $\beta$ such that $\det B = \lambda^{(J)}$.
We write $B = \Delta B_1$ where
$\Delta$ is the diagonal matrix with diagonal entries 
$(\lambda^{(J)}, 1, \ldots, 1)$ and $B_1$ has determinant~$1$.
By Proposition~\ref{prop:SLdense}, there exists a matrix
$A_1 \in \mathrm{SL}_d(K[\tau])$ such that
$$\big\Vert B_1 - A_1 \big\Vert 
< \min\left(\Vert B_1 \Vert, \, \frac{\Vert \lambda^{(J)} - 1 \Vert}{\Vert B_1\Vert^d \cdot \Vert \Delta \Vert}\right).$$
Then $\Vert A_1 \Vert = \Vert B_1 \Vert$.
Besides, expressing the inverse of $A_1$ using Cramer's formula, we
find $\Vert A_1^{-1} \Vert \leq \Vert A_1 \Vert^d = \Vert B_1 \Vert^d$.
Hence, if $I_d$
denotes the identity matrix of size $d$, we obtain
\begin{align*}
\big\Vert B A_1^{-1} - I_d \big\Vert 
 & = \big\Vert \Delta (B_1 - A_1) A_1^{-1} + (\Delta - I_d) \big\Vert \\
 & \leq \max\left(\big\Vert \Delta \big\Vert \cdot \big\Vert B_1 - A_1 \big\Vert \cdot \big\Vert A_1^{-1} \big\Vert, \, \big\Vert \lambda^{(J)} - 1 \big\Vert\right)
   = \big\Vert \lambda^{(J)} - 1 \big\Vert.
\end{align*}
Therefore, if $\alpha^{(J)} \in \mathcal E$ is the endomorphism whose
matrix in the basis $\mathcal B$ is $B A_1^{-1}$, we have
$$\alpha^{(J)}\big(\Lie_M(L[\tau])\big) = K \otimes_A U_M^{(J)}$$
given that $A_1^{-1}$ has coefficients in $K[\tau]$, which implies that the
endomorphism it represents induces a bijection of
$K \otimes_A \Lie_M(J[\tau]) = \Lie_M(L[\tau])$. Moreover,
$\Vert \alpha^{(J)} - \id_{\Lie_M(L_\infty\langle\tau\rangle)}\Vert 
\leq \Vert \lambda^{(J)} - 1 \Vert$,
proving that $\alpha^{(J)}$ converges to the identity when $J$ 
goes to $0$.
\end{proof}

\section{More on convergence}
\label{sec:convergence}

Throughout the article, we have worked with $L_\infty\langle\tau\rangle$
and $K_\infty\langle\tau\rangle$ which look perfectly suited for our
developments but actually do not capture the most accurate convergence
conditions one can expect in our context.
The aim of this final section is to show that one can replace them by
smaller spaces and derive from this strong convergence properties of
the $L$-series.

\subsection{Overconvergence of units}
\label{ssec:overconv}

We consider the space of \emph{overconvergent series} in the variable
$\tau$ defined as follows:
$$L_\infty\langle\tau\rangle^\dagger = \bigcup_{r > 1} L_\infty\langle\tau; r\rangle.$$
In other words, $L_\infty\langle\tau\rangle^\dagger$ consists of
skew power series that converge on a disc of radius strictly larger
than $1$. Clearly, $L_\infty\langle\tau\rangle^\dagger$ is a left
$R[\tau]$-module and so
$\Lie_M(L_\infty\langle\tau\rangle^\dagger)$ and
$E_M(L_\infty\langle\tau\rangle^\dagger)$ are well-defined.
We moreover have a natural inclusion
$$L_\infty\langle\tau\rangle^\dagger\subset L_\infty\langle\tau\rangle,$$
which induces natural inclusions
$$E_M\big(L_\infty\langle\tau\rangle^\dagger\big) \subset
  E_M\big(L_\infty\langle\tau\rangle\big) 
  \quad \text{and} \quad
  \Lie_M\big(L_\infty\langle\tau\rangle^\dagger\big) \subset
  \Lie_M\big(L_\infty\langle\tau\rangle\big).$$
We use them to endow all the above overconvergent
spaces with the inherited Gauss norm, that we continue to
denote by $\Vert \cdot \Vert$. We underline that the overconvergent
versions are no longer complete.

In the overconvergent setting, we nonetheless have an analogue of
Proposition~\ref{prop:expM}.

\begin{prop}
The exponential map $\exp_M$ induces by restriction and 
corestriction an $A[\tau]$-linear map
$$\exp_M : \Lie_M\big(L_\infty\langle\tau\rangle^\dagger\big) \longrightarrow
  E_M\big(L_\infty\langle\tau\rangle^\dagger\big)$$
which is uniformly locally an invertible isometry.
\end{prop}

The proof is identical to the proof of Proposition~\ref{prop:expM}
after the following lemma.

\begin{lem}
\label{lem:localhomeo}
There exists $\varepsilon > 0$ such that for all 
$x \in E_M(L_\infty\langle\tau\rangle^\dagger)$ with
$\Vert x \Vert \leq \varepsilon$, we have
$\log_M(x) \in \Lie_M(L_\infty\langle\tau\rangle^\dagger)$.
\end{lem}

\begin{proof}
Let $E$ be the Anderson module over $L$ corresponding to $M_L$.
We choose $\varepsilon$ such that $\log_E \in M_d(L_\infty\langle\tau; 2\varepsilon\rangle)$.
We also pick $r > 1$ such that $x \in E_M(L_\infty\langle\tau; r\rangle)$.
If $x = 0$, the lemma is clear. Otherwise, we consider its \emph{growth function}
defined by $G : (0, r) \to \RR_{> 0}$, $\rho \mapsto \Vert x \Vert_\rho$.
The function $\log G$ is a supremum of affine functions, so it is
continuous. Hence $G$ itself is continuous as well.
Besides, we know that $G(1) = \Vert x \Vert \leq \varepsilon$.
Therefore, there exists $r' > 1$ such that
$\Vert x \Vert_{r'} = G(r') \leq 2 \varepsilon$.
It then follows from Lemma~\ref{lem:rdcproduct} that
$\log_M(x) \in \Lie_M\big(L_\infty\langle\tau, r'\rangle\big)$
and we are done.
\end{proof}

For a nonzero ideal $J$ of $R$,
we can now define overconvergent versions of $U_M^{(J)}$ and $H_M^{(J)}$,
denoted by $U_M^{(J),\dagger}$ and $H_M^{(J),\dagger}$ respectively, as
the kernel and cokernel of the map
$$\Lie_M\big(L_\infty\langle\tau\rangle^\dagger\big)
  \xrightarrow{\exp_M}
  \frac{E_M\big(L_\infty\langle\tau\rangle^\dagger\big)}{E_M\big(J[\tau]\big)}
  \simeq E_M\big(L_\infty/J\langle\tau\rangle^\dagger\big).$$
We have natural $A[\tau]$-linear morphisms
$u_M^{(J)} : U_M^{(J),\dagger} \to U_M^{(J)}$ and
$h_M^{(J)} : H_M^{(J),\dagger} \to H_M^{(J)}$.

\begin{thm}
\label{thm:overconv}
The maps $u_M^{(J)}$ and $h_M^{(J)}$ are both isomorphisms.
\end{thm}

\begin{proof}
Applying the snake lemma to the following diagram
$$\begin{tikzcd}
  0 \ar[r]
& \Lie_M\big(L_\infty\langle\tau\rangle^\dagger\big) \ar[r] \ar[d,"\exp_M"]
& \Lie_M\big(L_\infty\langle\tau\rangle\big) \ar[r] \ar[d,"\exp_M"]
& \Lie_M\left(\frac{L_\infty\langle\tau\rangle}{L_\infty\langle\tau\rangle^\dagger}\right) \ar[r] \ar[d,"\alpha"]
& 0 \\
  0 \ar[r]
& E_M\big(L_\infty/J\langle\tau\rangle^\dagger\big) \ar[r]
& E_M\big(L_\infty/J\langle\tau\rangle\big) \ar[r]
& E_M\left(\frac{L_\infty\langle\tau\rangle}{L_\infty\langle\tau\rangle^\dagger}\right) \ar[r]
& 0
\end{tikzcd}$$
(where $\alpha$ is the induced map on cokernels),
we obtain the long exact sequence
$$0 \longrightarrow
  U_M^{(J),\dagger} \longrightarrow
  U_M^{(J)} \longrightarrow
  \ker \alpha \longrightarrow
  H_M^{(J),\dagger} \longrightarrow
  H_M^{(J)} \longrightarrow
  \coker \alpha \longrightarrow 0.$$
It is then enough to prove that $\alpha$ is an isomorphism.

We start by showing injectivity.
Let $x \in \Lie_M(L_\infty\langle\tau\rangle)$ such that
$\exp_M(x)$ is overconvergent. We want to prove that $x$ is itself
overconvergent.
We consider $\varepsilon > 0$ satisfying the condition of
Lemma~\ref{lem:localhomeo} together with the property that
$\exp_M$ is an isometry on the disc of radius~$\varepsilon$.
We write $x = x_1 + x_2$ with $x_1 \in \Lie_M(L[\tau])$
and $x_2 \in \Lie_M(L_\infty\langle\tau\rangle)$, 
$\Vert x_2 \Vert \leq \varepsilon$. Then
$\exp_M(x_1) \in E_M(L_\infty\langle\tau;\infty\rangle) \subset E_M(L_\infty\langle\tau\rangle^\dagger)$
and so
$$\exp_M(x_2) = \exp_M(x) - \exp_M(x_1) 
  \in E_M\big(L_\infty\langle\tau\rangle^\dagger\big).$$
Besides $\Vert \exp_M(x_2) \Vert = \Vert x_2 \Vert \leq
\varepsilon$. We then deduce from Lemma~\ref{lem:localhomeo}
that
$x_2 = \log_M(\exp_M(x_2)) \in \Lie_M(L_\infty\langle\tau\rangle^\dagger)$.
As a consequence, $x = x_1 + x_2 \in
\Lie_M(L_\infty\langle\tau\rangle^\dagger)$ as well and we conclude that $\alpha$ is injective.

We now move to surjectivity.
Let $y \in E_M(L_\infty\langle\tau\rangle)$. We
write $y = y_1 + y_2$ with $y_1 \in E_M(L[\tau])$ and
$y_2 \in E_M(L_\infty\langle\tau\rangle)$, $\Vert y_2
\Vert \leq \varepsilon$.
Then $\log_M(y_2) \in \Lie_M(L_\infty\langle\tau\rangle)$
and its image in 
$\Lie_M(L_\infty\langle\tau\rangle / L_\infty\langle\tau\rangle^\dagger)$
is a preimage of $y \bmod E_M(L_\infty\langle\tau\rangle^\dagger)$
by $\alpha$.
\end{proof}

\begin{rem}
The proof of the class formula presented in 
Sections~\ref{sec:motives}--\ref{sec:classformula}
extends almost \emph{verbatim} to the overconvergent
setting. Hence, if one prefers, one can totally avoid using
$L_\infty\langle\tau\rangle$ and work directly from the start
with overconvergent series.
\end{rem}

\subsection{Superconvergence of $L$-series}

Theorem~\ref{thm:overconv} has strong consequences on the
convergence properties of the $L$-series itself. Precisely,
mimicking the definition of $L_\infty\langle\tau\rangle^\dagger$,
we are led to consider the \emph{commutative} series
$\sum_i a_i \tau^i \in K_\infty\langle\tau\rangle$ for which there
exists $r > 1$ 
such that $\lim_{i \to \infty} \vert a_i \vert r^{q^i} = 0$.

Contrarily to the noncommutative case, those series not only
slightly overconverge outside the unit disc, but they actually
define entire functions on $K_\infty$, exhibiting extremely fast
convergence.
For this reason, we prefer using the term ``superconvergent''
in this case.

One should be careful that, although superconvergent series are
stable under sums, they are not stable under products. To handle
this slightly unpleasant situation, we introduce the following
definition.

\begin{deftn}[Superconvergence]
Let $\alpha$ be a positive real number.
A series $\sum_i a_i \tau^i \in K_\infty\langle\tau\rangle$ is
\emph{superconvergent at order $\alpha$}, or simply
\emph{$\alpha$-superconvergent}, if there exists $r > 1$ such that
$\lim_{i \to \infty} \vert a_i \vert r^{q^{\alpha i}} = 0$.
Their space is denoted by $K_\infty\llangle\tau; \alpha\rrangle$.
\end{deftn}

If $v_\infty$ denotes the $\infty$-adic valuation defined by
$v_\infty(a) = -\log_q \vert a \vert$ for $a \in K_\infty$,
one checks that $f = \sum_i a_i \tau^i$ is $\alpha$-superconvergent
if, and only if
$$\liminf_{i \to \infty} \: v_\infty(a_i) q^{-\alpha i} > 0.$$
We also note that $\alpha$-superconvergence implies $\beta$-superconvergence
for $\beta < \alpha$. Besides, $K_\infty\llangle\tau; \alpha\rrangle$
is a module over $K_\infty[\tau]$ for all $\alpha$.
However, as already underlined previously, it is not a ring.

\begin{lem}
\label{lem:multoverconv}
Let $f \in K_\infty\llangle\tau; \frac 1 \alpha\rrangle$ and
$g \in K_\infty\llangle\tau; \frac 1 \beta\rrangle$ for two positive
real numbers $\alpha$ and $\beta$. Then
$fg \in K_\infty\llangle\tau; \frac 1 {\alpha + \beta}\rrangle$.
\end{lem}

\begin{proof}
We write $f = \sum_i a_i \tau^i$ and $g = \sum_i b_j \tau^j$.
By Jensen's inequality applied to the convex function 
$x \mapsto q^x$ ($x \in \RR_{\geq 0}$), the points $\frac i \alpha, \frac j \beta$ and the weights
$\frac \alpha{\alpha+\beta}, \frac \beta{\alpha+\beta}$, we have
$$\alpha q^{i/\alpha} + \beta q^{j/\beta} \geq (\alpha + \beta) q^{(i+j)/(\alpha + \beta)}.$$
Hence, for nonnegative integers $i$ and $j$, we have
$$C_{i,j} := \frac 1 {\alpha{+}\beta} \cdot 
  \frac{v_\infty(a_i)+v_\infty(b_j)}{q^{(i+j)/(\alpha+\beta)}}
\geq\dfrac{ v_\infty(a_i)+v_\infty(b_j)}{\alpha q^{i/\alpha}+\beta q^{j/\beta}}$$
Writing $A_i := \alpha^{-1} v_\infty(a_i) q^{-i/\alpha}$ and
$B_i = \beta^{-1} v_\infty(b_i) q^{-i/\beta}$, we obtain
$$C_{i,j} \geq \frac{A_i \alpha q^{i/\alpha} + B_j \beta q^{j/\beta}}{\alpha q^{i/\alpha} + \beta q^{j/\beta}}.$$
In other words, $C_{i,j}$ is controlled by a weighted mean of $A_i$ and
$B_j$ with weights going to infinity when $i$ and $j$ grow. It follows that 
$$\liminf_{i+j \to \infty}C_{i,j} \geq \min\big( \liminf_{i \to \infty} A_i, \liminf_{j \to \infty} B_j\big) > 0.$$
which shows that $fg = \sum_{i,j} a_i b_j \tau^{i+j} \in K_\infty\llangle\tau; \frac 1{\alpha+\beta}\rrangle$.
\end{proof}

\begin{cor}
\label{cor:Lsuperconv}
We have
$L^{(J)}(M; \tau) \in K_\infty\llangle\tau; \frac 1 d\rrangle$ with
$d = \dim_K \Lie_M(L)$.
\end{cor}

\begin{proof}
Let $u_1,\ldots,u_d\in U_M^{(J)}$ be units that form a basis of $K[\tau]\otimes_{A[\tau]}U_M^{(J)}$ over $K[\tau]$, and write 
\[
u_n = \sum_{i}{\ell_{n,i} \tau^i} \quad \text{for some} \quad \ell_{n,i}\in \Lie_M(L_{\infty}).
\]
Given a basis $(e_1,\ldots,e_d)$ of $\Lie_M(L)$ over $K$, each $\ell_{n,i}$ decomposes as $\sum_{j}{c_{n,i,j}e_j}$ for some coefficients $c_{n,i,j}\in K_{\infty}$. After base change to $K[\tau]$, the class formula (Theorem~\ref{thm:classformula}) expresses $L^{(J)}(M;\tau)$ as the determinant of the $d$-times-$d$ matrix $T=(\sum_{i}{c_{n,i,j} \tau^i})_{n,j}$ in $M_d(K_{\infty}\langle \tau \rangle)$, at least up to a factor in $K^{\times}$. From overconvergence of units (Theorem~\ref{thm:overconv}), there exists $r>1$ such that $\vert \ell_{n,i}\vert r^{q^i}$ tends to zero as $i$ goes to infinity. We deduce the same estimate for the coefficients $c_{n,i,j}$:
\[
\text{for all } n \text{ and } j, \quad \lim_{i \to \infty} \vert c_{n,i,j}\vert r^{q^i} =0. 
\]
In particular, each entry of $T$ is $1$-superconvergent and $\det(T)$ is $\frac{1}{d}$-superconvergent by Lemma~\ref{lem:multoverconv}.
\end{proof}

Corollary~\ref{cor:Lsuperconv} resonates strongly with
\cite[Theorem~2.2.6]{caruso-gazda} except that the rank is here
replaced by the dimension.
In order to conceal those two results, we propose the following
conjecture, perfectly in line with Taelman's conjecture~\cite{taelman-conj}.

\begin{conj}
We have
$L^{(J)}(M; \tau) \in K_\infty\llangle\tau; \frac 1 w\rrangle$ with
$w = \dim_K \Lie_M(L) / \mathfrak j{\cdot}\Lie_M(L)$ 
where $\mathfrak j$ is the
ideal of $A \otimes R$ generated by the $a \otimes 1 - 1 \otimes
a$, $a \in A$.
\end{conj}

\begin{rem}
The dimension $w$ is upper bounded by both
the dimension and the rank of $M$ (if abelian). Our conjecture is
then a strong version of both Corollary~\ref{cor:Lsuperconv} and
Theorem~2.2.6 of \cite{caruso-gazda}.
\end{rem}


{\footnotesize
\begin{thebibliography}{ABCD}

\bibitem[A1]{anderson}
G.~W. Anderson,
\emph{$t$-motives},
Duke Math. J. \textbf{53} (1986), no.~2, 457--502.
\href{https://doi.org/10.1215/S0012-7094-86-05328-7}
{doi:10.1215/S0012-7094-86-05328-7}.

\bibitem[A2]{anderson-L}
G.~W. Anderson,
\emph{An elementary approach to $L$-functions mod $p$},
J. Number Theory \textbf{80} (2000), no.~2, 291--303.
\href{https://doi.org/10.1006/jnth.1999.2452}
{doi:10.1006/jnth.1999.2452}.

\bibitem[AT]{AT}
B.~Anglès and F.~Tavares Ribeiro,
\emph{Arithmetic of function field units},
Math. Ann. \textbf{367} (2017), no.~1--2, 501--579.
\href{https://doi.org/10.1007/s00208-016-1405-2}
{doi:10.1007/s00208-016-1405-2}.

\bibitem[ANT]{ANT}
B.~Anglès, T.~Ngo Dac, and F.~Tavares Ribeiro,
\emph{A class formula for admissible Anderson modules},
Invent. Math. \textbf{229} (2022), no.~2, 563--606.
\href{https://doi.org/10.1007/s00222-022-01110-3}
{doi:10.1007/s00222-022-01110-3}.

\bibitem[FMM]{CFM}
M.~I. de Frutos-Fernández, D.~Macías Castillo, and
D.~Martínez Marqués,
\emph{The refined class number formula for Drinfeld modules},
Compos. Math. \textbf{162} (2026), no.~4, 852--903.
\href{https://doi.org/10.1017/S0010437X26103108}
{doi:10.1017/S0010437X26103108}.

\bibitem[CG]{caruso-gazda}
X.~Caruso and Q.~Gazda,
\emph{Computation of classical and $v$-adic $L$-series of
$t$-motives},
Res. Number Theory \textbf{11} (2025), article no.~35.
\href{https://doi.org/10.1007/s40993-024-00588-5}
{doi:10.1007/s40993-024-00588-5}.

\bibitem[Fa]{fang}
J.~Fang,
\emph{Special $L$-values of abelian $t$-modules},
J. Number Theory \textbf{147} (2015), 300--325.
\href{https://doi.org/10.1016/j.jnt.2014.07.012}
{doi:10.1016/j.jnt.2014.07.012}.

\bibitem[Fe]{ferraro}
G.~H. Ferraro,
\emph{A duality result about special functions for Drinfeld
modules of arbitrary rank},
Res. Math. Sci. \textbf{12} (2025), article no.~23.
\href{https://doi.org/10.1007/s40687-025-00506-w}
{doi:10.1007/s40687-025-00506-w}.

\bibitem[G]{gazda}
Q.~Gazda,
\emph{On the integral part of $A$-motivic cohomology},
Compos. Math. \textbf{160} (2024), no.~8, 1715--1783.
\href{https://doi.org/10.1112/S0010437X24007218}
{doi:10.1112/S0010437X24007218}.

\bibitem[GM]{gazda-maurischat}
Q.~Gazda and A.~Maurischat,
\emph{Pairing Anderson motives via formal residues in the
Frobenius endomorphism},
preprint, arXiv:2504.01926 [math.AG] (2025).
\href{https://doi.org/10.48550/arXiv.2504.01926}
{doi:10.48550/arXiv.2504.01926}.

\bibitem[Ha]{hartl}
U.~Hartl,
\emph{Isogenies of abelian Anderson $A$-modules and
$A$-motives},
Ann. Sc. Norm. Super. Pisa Cl. Sci. (5) \textbf{19} (2019),
no.~4, 1429--1470.
\href{https://doi.org/10.2422/2036-2145.201612_003}
{doi:10.2422/2036-2145.201612-003}.

\bibitem[La]{lam}
T.~Y. Lam,
\emph{Lectures on Modules and Rings},
Graduate Texts in Mathematics, vol.~189,
Springer-Verlag, New York, 1999.
\href{https://doi.org/10.1007/978-1-4612-0525-8}
{doi:10.1007/978-1-4612-0525-8}.

\bibitem[Qu]{quillen}
D.~Quillen,
\emph{Projective modules over polynomial rings},
Invent. Math. \textbf{36} (1976), no.~1, 167--171.
\href{https://doi.org/10.1007/BF01390008}
{doi:10.1007/BF01390008}.

\bibitem[Ta1]{taelman-conj}
L.~Taelman,
\emph{Special $L$-values of $t$-motives: a conjecture},
Int. Math. Res. Not. IMRN (2009), no.~16, 2957--2977.
\href{https://doi.org/10.1093/imrn/rnp038}
{doi:10.1093/imrn/rnp038}.

\bibitem[Ta2]{taelman}
L.~Taelman,
\emph{Special $L$-values of Drinfeld modules},
Ann. of Math. (2) \textbf{175} (2012), no.~1, 369--391.
\href{https://doi.org/10.4007/annals.2012.175.1.10}
{doi:10.4007/annals.2012.175.1.10}.

\end{thebibliography}
}
\end{document}